\documentclass[12pt]{amsart}
\usepackage{amsfonts}
\usepackage{graphicx}
\usepackage{pgf,tikz}
\usetikzlibrary{arrows}
\usepackage[letterpaper, left=2.5cm, right=2.5cm, top=2.5cm,
bottom=2.5cm,dvips]{geometry}
\usepackage[T1]{fontenc} 

\newtheorem{theorem}{Theorem}
\theoremstyle{plain}

\numberwithin{equation}{section}

\begin{document}
\title[Brooks' theorem]{A short proof of Brooks' theorem}

\author{Mariusz Zaj\k{a}c}
\address{Faculty of Mathematics and Information Science, Warsaw University
	of Technology, Koszykowa 75, 00-662 Warsaw, Poland}
\email{M.Zajac@mini.pw.edu.pl}

\begin{abstract}
We give a simple short proof of Brooks' theorem using only induction and greedy coloring, while avoiding issues of graph connectivity. The argument generalizes easily to some extensions of Brooks' theorem, including its variants for list coloring, signed graphs coloring and correspondence coloring.
\end{abstract}

\maketitle

\section{Introduction}
Brooks' theorem \cite{Brooks} is a classic result in graph coloring. It asserts that a connected graph of maximum degree at most $k$ is $k$-colorable, unless it is a clique or an odd cycle. Several substantially different proofs of this theorem are currently known (see the survey \cite{CranstonRabern}). Among the most elegant are the one due to Lov\'{a}sz \cite{Lovasz} and its modifications (e.g. \cite{Bondy}).

In this note we give a short and self-contained proof of Brooks' theorem. Its basic idea is similar to that of \cite{Lovasz}, although the connectivity issues are avoided. The argument can be easily extended to some known generalizations of Brooks' theorem, including the recent results on signed graphs \cite{MacajovaRS}, \cite{Stiebitz} and the more general correspondence coloring \cite{BernshteynKP}, \cite{DvorakPostle}.

\section{The proof}

We will make use of the following obvious observation. Let $G$ be a graph whose maximum degree satisfies $\Delta (G)\leq k$. Suppose that $G$ is partially colored using at most $k$ colors. Let $P=v_1v_2\dots v_j$ be a path in $G$, and assume that the vertices of $P$ are uncolored. Then we may color all vertices from $v_1$ up to $v_{j-1}$ consecutively along $P$, since at the moment of coloring the vertex $v_i$ its neighbor $v_{i+1}$ is yet uncolored. We denote this sequential coloring procedure by {\sc PathColor}$(v_1,v_2,\dots,v_{j-1};v_j)$. Note that after its execution the last vertex $v_j$ of the path $P$ remains uncolored, in particular {\sc PathColor} does nothing if $j=1$.

\begin{theorem}\emph{(Brooks \cite{Brooks})} \label{Theorem Brooks}
Let $k\geq 3$ be a natural number. Let $G$ be a graph with $\Delta(G)\leq k$. If $G$ does not contain a clique on $k+1$ vertices, then $G$ is $k$-colorable.
\end{theorem}

\begin{proof}
The proof is by induction on the number $n$ of vertices of $G$. For $n \leq k$ the assertion holds trivially.  We may further assume that $G$ is $k$-regular, since otherwise we would delete a vertex of degree strictly less than $k$ and apply induction. Let $v$ be any vertex of $G$. Since $G$ does not contain a clique on $k+1$ vertices, there exist two neighbors $x,y$ of $v$ that are not adjacent in $G$. Denote $v_1=x$, $v_2=v$, and $v_3=y$. Let $P=v_1v_2v_3\dots v_r$ be a path starting with these three vertices and extending itself maximally, i.e. until some vertex $v_r$ whose all neighbors are already on $P$.

\textit{Case 1.} Suppose first that $r=n$, which means that $P$ contains all vertices of $G$, and let $v_j$ be any neighbor of $v_2$ other than $v_1$ and $v_3$ (it exists since $k\geq 3$). At first, give the vertices $v_1$ and $v_3$ the same color. Then apply procedures {\sc PathColor}$ (v_4,v_5,\dots,v_{j-1};v_j)$ and {\sc PathColor}$(v_n,v_{n-1},\dots,v_j;v_2)$. Finally, color the vertex $v_2$, which is possible because it has two neighbors in the same color. The entire graph $G$ is now colored.

\textit{Case 2.} Assume now that $r<n$. Recall that all neighbors of $v_r$ are on the path $P$. Let $v_j$ be the neighbor of $v_r$ with the smallest index. So, $C=v_jv_{j+1}\dots v_r$ is a cycle in $G$. Consider the subgraph $G'=G-C$ obtained by deleting all vertices of $C$. At first, color $G'$ using $k$ colors by the induction hypothesis. If there is no edge between $G'$ and $C$, then we are done by applying induction also to the subgraph induced by $C$. If, on the contrary, there is a vertex on $C$ with a neighbor in $G'$, then let $v_l$ be such vertex with the largest index, and let $u$ be any of its neighbors in $G'$. Notice that $l<r$ because $v_r$ has all of its neighbors on $C$. Since the vertex $v_{l+1}$ does not have neighbors in $G'$, we may assign it the same color as $u$. Now apply procedure {\sc PathColor}$(v_{l+2},\dots,v_r,v_j,\dots,v_{l-1};v_l)$ and finally color $v_l$, which is possible as it has two neighbors in the same color. As previously, the entire graph $G$ is colored and the proof is complete.
\end{proof}

The above proof can be readily converted into a coloring algorithm, with the \emph{depth-first search} (DFS) method used to
emulate the inductive argument. In order to streamline the implementation we may consider the following points:
 
\begin{itemize}
\item 
First, there is no actual need to check if the graph $G$ is regular. Contrariwise, if the vertex $v$ has degree strictly less than $k$, then we
can build a DFS tree with root at $v$ and color the vertices of $G$ greedily in the order of decreasing DFS labels.

\item
In Case 2 we do not stop to invoke the inductive assumption when we reach the end of the path $P$. Instead, we
let the DFS algorithm run till it labels all the vertices, and color greedily $v_n,v_{n-1}, \dots, v_{r+1}$
and $v_1,v_2, \dots, v_{j-1}$ before coloring $C$.

\item
If $\deg(v_r)=1$, then $C$ is a $K_2$ graph rather than a cycle, but then the last part of the proof works 
trivially ({\sc PathColor} has to color 0 vertices). Alternatively, we can restart the DFS algorithm 
with $v_r$ as the root.
\end{itemize}

\section{Generalizations}

Suppose that each vertex $v$ of a graph $G$ is given an arbitrary \emph{list} $L(v)$ of $k$ admissible colors. A \emph{list coloring} \cite{ErdosRT} of $G$ is a proper coloring in which color of any vertex is chosen from its list. Notice that the proof of Theorem \ref{Theorem Brooks} extends easily to the list coloring variant of Brooks' theorem. Indeed, {\sc PathColor} works exactly the same since each list, as previously, has $k$ colors. Furthermore, for the first step in Case 1 we may choose a pair of colors for $v_1$ and $v_3$ blocking at most one color from $L(v_2)$, while in Case 2 way we can analogously find such a color for $v_{l+1}$, that $v_{l+1}$ and $u$ block at most one color from $L(v_l)$.

Using the data structure provided by Skulrattanakulchai in the proof of Theorem 12 in \cite{Sku}, we can easily see that the coloring according
to our algorithm takes O$(m+n)$ time, where $m$ is the number of edges in $G$. The present algorithm cannot therefore be considerably 
faster than the one shown in \cite{Sku}, but it is arguably simpler both in terms
of the proof of correctness and the implementation, as it neither relies on finding a biconnected component of $G$ nor looks 
for any special subgraphs like cliques or theta graphs. 

Actually, our argument gives a similar result in a more general setting of \emph{correspondence colorings}, introduced recently by Dvo\v{r}\'{a}k and Postle in \cite{DvorakPostle}. Given a graph $G$ and a list assignment $L(v)$ we may define a \emph{constraint graph} $F=F(L,G)$ as follows. The vertices of $F$ are all the pairs $(v,a)$ with $a\in L(v)$. 
Now for any vertices $u$ and $v$ adjacent in $G$ we fix an arbitrary (possibly partial) matching between
$\{ u \} \times L(u)$ and $\{ v \} \times L(v)$; there are no other edges in $F$. 
 
A coloring $c$ of $G$ from lists $L(v)$ \emph{respects} $F$ if $(u,c(u))$ and $(v,c(v))$ are never adjacent in $F$. 
For instance, if $(u,a)$ is adjacent to $(v,b)$ in $F$ only when $a=b$, then a coloring of $G$ respecting $F$ is just the usual list coloring of $G$.

Notice that the matching condition in the definition of the constraint graph $F$ guarantees that if we color a single vertex $v$,
then the list of colors still available at any $u$ adjacent to $v$ is reduced at most by one element. Consequently,
{\sc PathColor} may be applied just as above for constructing a coloring respecting $F$. Hence we get the 
following result mentioned in \cite{DvorakPostle}.

\begin{theorem}
Let $k\geq 3$ be a natural number, and let $G$ be a graph with $\Delta(G)\leq k$ not containing a clique on $k+1$ vertices. Suppose that $L(v)$ is an arbitrary list assignment, with $\lvert L(v)\rvert =k$. Let $F$ be an arbitrary constraint graph for $(L,G)$. Then there is a list coloring of $G$ respecting $F$.
\end{theorem}


It is perhaps worth mentioning that the above theorem covers an interesting special case of signed graphs. A \emph{signed graph} is a graph $G$ whose edges are labeled arbitrarily by the elements of the set $\{+,-\}$. A vertex coloring $c$ of a signed graph by real numbers is \emph{proper} if for every pair $u,v$ of adjacent vertices, $c(u)+c(v)\neq0$ or $c(u)-c(v)\neq0$, according to the label of the edge $uv$. This is clearly a special case of correspondence coloring. So, Theorem 2 extends some Brooks type theorems for signed graphs proved in \cite{MacajovaRS}, \cite{FleinerWiener} 
and \cite{Stiebitz}.


\begin{thebibliography}{99}

\bibitem{BernshteynKP} A. Bernshteyn, A. Kostochka, S. Pron, On DP-coloring of graphs and multigraphs, Siberian Math. J. 58 (2017) 28--36.

\bibitem{Brooks} R. L. Brooks, On colouring the nodes of a network, Math. Proc. Cambridge Philos. Soc. 37 (1941) 194--197.

\bibitem{Bondy} J. Bondy, Short proofs of classical theorems, J. Graph Theory 44 (2003) 159--165.


\bibitem{CranstonRabern} D. W. Cranston, L. Rabern, Brooks' theorem and beyond, J. Graph Theory 80 (2015) 199--225.

\bibitem{DvorakPostle} Z. Dvo\v{r}\'{a}k, L. Postle, Correspondence coloring and its application to list-coloring planar graphs without cycles of lengths 4 to 8, J. Combin. Theory Ser. B (2017) 
http://dx.doi.org/10.1016/j.jctb.2017.09.001

\bibitem{ErdosRT} P. Erd\H{o}s, A. Rubin, H. Taylor, Choosability in graphs, in: Proc. West Coast Conf. on Combinatorics, Graph Theory and Computing, Congressus Numerantium 26 (1979) pp. 125{--}157.

\bibitem{FleinerWiener} T. Fleiner, G. Wiener, Coloring signed graphs using DFS, Optim. Lett.  (2016) 10:865--869

\bibitem{Lovasz} L. Lov\'{a}sz, Three short proofs in graph theory, J. Combin. Theory Ser. B 19(3) (1975) 269{--}271.

\bibitem{MacajovaRS} E. M\'{a}\v{c}ajov\'{a}, A. Raspaud, M. \v{S}koviera, The chromatic number of signed graphs, Electronic J. Combin. 23(1) (2016), P1.14.

\bibitem{Stiebitz} T. Schweser, M. Stiebitz, Degree choosable signed graphs, Discrete Math. 340 (2017) 882--891.

\bibitem{Sku}
S. Skulrattanakulchai, {$\Delta$}-list vertex coloring in linear time, Inform. Process. Lett. 98 (2006) no. 3, 101--106

\end{thebibliography}
\end{document}